\documentclass[twoside,a4,11pt]{amsart}

\usepackage[english]{babel}
\usepackage[utf8x]{inputenc}
\usepackage{amsmath}
\usepackage{graphicx}
\usepackage[colorinlistoftodos]{todonotes}
\usepackage{amsmath,amscd,amsthm}
\usepackage{amstext, amsfonts, a4}
\usepackage{amssymb}
\usepackage{tikz-cd}
\usepackage{fancyhdr}
\usepackage{enumerate}
\usepackage{cleveref}
\usepackage{xcolor}

\numberwithin{equation}{section}
\newtheorem{thm}[equation]{Theorem}

\newtheorem{prop}[equation]{Proposition}
\newtheorem{cor}[equation]{Corollary}
\newtheorem{lem}[equation]{Lemma}

\theoremstyle{definition}
\newtheorem{exmp}[equation]{Example}

\newtheorem{qu}[equation]{Question}
\newtheorem{rem}[equation]{Remark}

\theoremstyle{plain}

\renewcommand{\deg}{\operatorname{\mathsf{deg}}}
\newcommand\ind{\operatorname{\mathsf{ind}}}
\renewcommand\exp{\operatorname{\mathsf{exp}}}
\renewcommand\log{\operatorname{\mathsf{log}}}
\renewcommand{\setminus}{\smallsetminus}
\newcommand\op{\operatorname{\mathsf{op}}}

\newcommand\Br{\operatorname{\mathsf{Br}}}

\newcommand\Frob{\operatorname{\mathsf{Frob}}}
\newcommand\cchar{\operatorname{\mathsf{char}}}

\newcommand\N{\operatorname{\mathsf{N}}}
\newcommand\CH{\operatorname{\mathsf{CH}}}

\renewcommand{\leq}{\leqslant}
\renewcommand{\geq}{\geqslant}

\begin{document}
\title{Symbol length in positive characteristic}
	
\date{\today}
	
\author{Fatma Kader B\.{i}ng\"{o}l}
	
\address{Universiteit Antwerpen, Departement Wiskunde, Middelheim\-laan~1, 2020 Ant\-werpen, Belgium.}
\email{FatmaKader.Bingol@uantwerpen.be}
		
\begin{abstract}
    We show that any central simple algebra of exponent $p$ in prime characteristic $p$ that is split by a $p$-extension of degree $p^n$ is Brauer equivalent to a tensor product of $2\cdot p^{n-1}-1$ cyclic algebras of degree $p$. 
    If $p=2$ and $n\geq3$, we improve this result by showing that such an algebra is Brauer equivalent to a tensor product of $5\cdot2^{n-3}-1$ quaternion algebras. 
    Furthermore, we provide new proofs for some bounds on the minimum number of cyclic algebras of degree $p$ that is needed to represent Brauer classes of central simple algebras of exponent $p$ in prime characteristic $p$, which have previously been obtained by different methods.
		
    \medskip\noindent
    {\sc Classification} (MSC 2020): 16K20, 13A35
		
    \medskip\noindent
    {\sc{Keywords:}} cyclic algebra, symbol length, positive characteristic
\end{abstract}
	
\maketitle

\section{Introduction}
Let $F$ be a field. We set $\mathbb{N}_+=\mathbb{N}\setminus\{0\}$.
Let $n\in\mathbb{N}_+$, $L/F$ be a cyclic field extension of degree $n$ and let $\sigma$ be a generator of its Galois group. 
For a given $b\in F^{\times}$, the rules
\begin{equation*}
    j^n=b\quad\text{and}\quad xj=j\sigma(x)\quad \text{for all}\quad x\in L
\end{equation*}
determine on the $L$-vector space $L\oplus jL\oplus\ldots\oplus j^{n-1}L$ a ring multiplication turning it into a central simple $F$-algebra of degree $n$, which is denoted by $$[L/F,\sigma,b).$$   
Any central simple $F$-algebra of degree $n$ containing a cyclic field extension $L/F$ of degree $n$ is isomorphic to $[L/F,\sigma,b)$ for certain $b\in F^{\times}$; see \cite[Theorem 5.9]{Alb39}.
Such an algebra is called a \emph{cyclic algebra}. 
We denote by $\Br(F)$ the Brauer group of $F$, and for $n\in\mathbb{N}_+$, we denote by $\Br_n(F)$ the $n$-torsion part of $\Br(F)$.

\begin{thm}[A. A. Albert, A. S. Merkurjev, A. A. Suslin \cite{Alb39}, \cite{MS1982}]\label{Albert-Merkur-Suslin}
    Let $n\in\mathbb{N}_+$. Assume $F$ contains a primitive $n$-th root of unity or $n$ is a power of $\cchar F$. 
    Then $\Br_{n}(F)$ is generated by the classes of cyclic algebras of degree $n$.
\end{thm}

Given a central simple $F$-algebra $A$, we denote by $\deg A$, $\ind A$ and $\exp A$, the degree, index and exponent of $A$, respectively.
For $n\in\mathbb{N}_+$, we denote by $A^{\otimes n}$ the $n$-fold tensor product $A\otimes_F\ldots\otimes_FA$.
\begin{rem}\label{generic-indecomp-Karpenko}
    Let $p$ be a prime number.
    In \cite{SchBerg1992}, a generic division algebra $D(p^n,p^e)$ of degree $p^n$ and exponent $p^e$, with $n,e\in\mathbb{N}$, $e\leq n$, is constructed as follows: 
    let $D$ be a division algebra over a field $F$ (of any characteristic) of degree and exponent equal to $p^n$. 
    Let $K$ be the function field of the Severi-Brauer variety corresponding to $D^{\otimes p^e}$ and set $D(p^n,p^e)=D\otimes_FK$. 
    For $e\geq2$, it is shown that $\ind D(p^n,p^e)^{\otimes p}=p^{n-1}$. 
    This implies in particular that $D(p^n,p^e)$ is indecomposable for any $n\in\mathbb{N}$ and $e\geq2$.
    In \cite{Karp1995}, N. Karpenko shows that $D(p^n,p)$ is indecomposable for any $n\in\mathbb{N}$ except when $p=n=2$.
\end{rem}

From these results the study of the so-called symbol length arose. 
For a central simple $F$-algebra $A$, the \emph{$n$-symbol length of} $A$, denoted by $\lambda_{n}(A)$, is the smallest $m\in\mathbb{N}$ such that $A$ is Brauer equivalent to a tensor product of $m$ cyclic algebras of degree $n$.
We set $\lambda_n(A)=\infty$ if $A$ is not Brauer equivalent to a tensor product of cyclic algebras of degree $n$.

Note that, if $n\in\mathbb{N}_+$ and $A$ is a central simple $F$-algebra with $\exp A$ dividing $n$, one always has $\lambda_n(A)\geq\log_n\ind A$.
The challenge is to bound the $n$-symbol length in terms of the index.  
We have the following results for algebras of exponent $2$ in small index.
\begin{thm}\label{Albert-Tignol-Rowen-2-symbol-small degree}
    Let $A$ be a central simple $F$-algebra with $\exp A=2$. Let $n\in\mathbb{N}_+$ be such that $\ind A=2^n$.
\begin{enumerate}[$(1)$]
    \item \textnormal{(Albert \cite{Albert32})} If $n\leq2$, then $\lambda_2(A)=n$.
    \item \textnormal{(J. -P. Tignol--L. Rowen \cite{Tig78}, \cite{Row84})} If $n=3$, then $3\leq\lambda_2(A)\leq4$.
\end{enumerate}	
\end{thm}

In the case where $n$ is a power of the characteristic, we know more about the $n$-symbol length. 
\begin{thm}[M. Florence \cite{Flo13}]\label{Flo-p-symbol} 
    Assume that $\cchar F=p$ and let $e,n\in\mathbb{N}_+$. Let $A$ be a central simple $F$-algebra with $\exp A=p^e$ and $\ind A=p^n$. 
    Then $\lambda_{p^e}(A)\leq p^n-1$.	
\end{thm} 

Florence's argument involves a generic splitting field given by the function field of the Severi-Brauer variety attached to the algebra.
Previously, P. Mammone and A. Merkurjev \cite{MamMerkur1991} had obtained  better bounds for the special case of cyclic algebras. 
Their argument is based on a computation of the corestriction of a cyclic algebra.

\begin{thm}[Mammone--Merkurjev]\label{MM-cyclic-p-symbol}
    Assume that $\cchar F=p$. Let $e,n\in\mathbb{N}_+$. Let $C$ be a cyclic $F$-algebra with $\exp A=p^e$ and $\deg C=p^n$. 
    Then $\lambda_{p^e}(A)\leq p^{n-e}$.
\end{thm} 

In \Cref{Albert-Tignol-Rowen-2-symbol-small degree}, for the case of $n=3$, the existence of a subfield of the underlying division algebra which is a Galois extension with Galois group $(\mathbb{Z}/2\mathbb{Z})^3$ due to Rowen \cite{Row78}, \cite{Row84} plays a crucial role. 
More generally, \Cref{Albert-Merkur-Suslin} implies that any central simple algebra of exponent $p$, over a field of characteristic $p$ or containing a primitive $p$-th root of unity, admits a splitting field which is a Galois extension with Galois group $(\mathbb{Z}/p\mathbb{Z})^n$ for some $n\in\mathbb{N}_+$.
Hence it is natural to consider the problem to determine the $p$-symbol length for such a family of algebras.  

\begin{qu}\label{p-symbol-p-elemantary-crossed-product?}
    Let $A$ be a central simple $F$-algebra with $\exp A=p$.
    Let $n\in\mathbb{N}_+$ and assume that there exists a Galois extension with Galois group $(\mathbb{Z}/p\mathbb{Z})^n$ that splits $A$.
    What is the $p$-symbol length of $A$?
\end{qu}

If $p=2$ and $n\leq3$, then \Cref{Albert-Tignol-Rowen-2-symbol-small degree} answers \Cref{p-symbol-p-elemantary-crossed-product?}.
Let us provide other known results in this context.
\begin{thm}\label{Sivatski-Eli-symbol-crossed-product}
    Let $A$ be a central simple $F$-algebra with $\exp A=p$.
    Let $n\in\mathbb{N}_+$ and assume that there exists a Galois extension with Galois group $(\mathbb{Z}/p\mathbb{Z})^n$ that splits $A$.
\begin{enumerate}[$(1)$]
    \item \textnormal{(A. S. Sivatski \cite{Siv2013})} If $p=2$, $n=4$, and $\cchar F\neq 2$, then $\lambda_2(A)\leq18$. 
    \item \textnormal{(E. Matzri \cite{Matz2014})} If $p=3$, $n=2$ and $F$ contains a primitive $3$rd root of unity, then $\lambda_3(A)\leq 31$.
\end{enumerate}    
\end{thm}

We call a finite field extension $L/F$ a \emph{$p$-extension} if there exist $n\in\mathbb{N}$ and a tower of fields $(L_i)_{i=0}^{n}$ with $L_0=F$, $L_n=L$, and where $L_i/L_{i-1}$ is a cyclic extension of degree $p$ for $i=1,\ldots,n$.
Clearly, a Galois extension with Galois group $(\mathbb{Z}/p\mathbb{Z})^n$ is a $p$-extension.

We consider \Cref{p-symbol-p-elemantary-crossed-product?} over fields of characteristic $p$ (\Cref{section:Galois-p-elemantary-split}).
We investigate the behaviour of $p$-symbol length under finite field extension of various kind (\Cref{multipurinsep-ind-reduce to p+cyclic}, \Cref{behavior-symbol-length-pur-insep-deg-p}, \Cref{behavior-symbol-length-pur-insep}, \Cref{behavior-p-symbol-length-cyclic-ext-deg=p}, \Cref{behavior-p-symbol-length-p-ext}).
In prime characteristic $p$, we show that the $p$-symbol length of any central simple algebra of exponent $p$ that is split by a $p$-extension of degree $p^n$ is bounded by $2\cdot p^{n-1}-1$ (\Cref{p-symbol-length-Galois-(Z/2)^p}). 
If $p=2$ and $n\geq3$, we improve this bound by showing that the $2$-symbol length of such an algebra is bounded above by $5\cdot2^{n-3}-1$ (\Cref{2-symbol-length-Galois-(Z/2)^n}).
Furthermore, we obtain new and more elementary proofs for \Cref{Flo-p-symbol} and \Cref{MM-cyclic-p-symbol} above in certain cases.
In \Cref{2-symbol-length-char=2}, we retrieve the bounds from \Cref{Flo-p-symbol} for $p=2$ and $e=1$, and in \Cref{symbol-length-cyclic-exp=p-char=p}, we retrieve the bounds from \Cref{MM-cyclic-p-symbol} for $e=1$.
Our arguments are based on the study of purely inseparable splitting fields and the Frobenius morphism. 

The following two sections serve as preliminaries. For the theory of central simple algebras, we refer to \cite{Alb39} and \cite{Jacob1996}.

\section{Some classical results}
Let $p$ be always a prime number. We assume throughout the paper that $\cchar F=p$. Let $L/F$ be a cyclic extension of degree $p$. 
Then $L=F(i)$ with $i^p-i=a\in F$ for some $i\in L\setminus F$. 
We denote the cyclic algebra $[L/F,\sigma,b)$ where $\sigma$ is a generator of the Galois group of $L/F$ and $b\in F^{\times}$ simply by $[a,b)_{p,F}$. 
Letting further $b=j^p$, the algebra $[a,b)_{p,F}$ is generated by $i$ and $j$, which are subject to the rules
$$i^p-i=a,\quad j^p=b,\quad ji=(i-1)j.$$ 
Furthermore, any cyclic $F$-algebra of degree $p$ is of the form $[a,b)_{p,F}$ for some $a\in F$ and $b\in F^{\times}$. 

For a finite field extension $K/F$, let $\N_{K/F}:K\to F$ denote the norm map.
For two central simple $F$-algebras $A$ and $B$, we write $A\sim B$ to indicate that $A$ and $B$ are Brauer equivalent.
\begin{prop}\label{cyclic-alg-properties}
    Let $a\in F$and $b,b'\in F^{\times}$. The following hold: 
\begin{enumerate}[$(1)$]
    \item $[a,b)_{p,F}\otimes_F[a,b')_{p,F}\sim[a,bb')_{p,F}$.
    \item $[a,b)_{p,F}$ is split if and only if $b\in\N_{F(i)/F}(F(i)^{\times})$ where $i^p-i=a$.
\end{enumerate}	  
\end{prop}
\begin{proof}
    See \cite[Theorem 5.11 and Theorem 5.14]{Alb39}.
\end{proof} 
 
Purely inseparable field extensions occur as splitting fields for central simple $F$-algebras of $p$-power exponent. 
These extensions are useful tools in the study of the symbol length.
\begin{thm}[Albert]\label{Albert-purins->decomp-palg} 
    Let $A$ be a central simple $F$-algebra. 
    Let $n\in\mathbb{N}_+$ and $b_1,\ldots,b_n\in F^{\times}$ be such that the field $K=F(\sqrt[p]{b_1},\ldots,\sqrt[p]{b_n})$ splits $A$. 
    There exist some $a_1,\ldots,a_n\in F$ such that $$A\sim[a_1,b_1)_{p,F}\otimes_F\ldots\otimes_F[a_n,b_n)_{p,F}.$$
\end{thm}
\begin{proof}
    See \cite[Theorem 7.28]{Alb39}.
\end{proof}

\section{The frobenius morphism}
Let $K/F$ be a purely inseparable field extension such that $K^p\subseteq F$. 
The map $\Frob_{K/F}: K\to F$, $x\mapsto x^p$ is called the \emph{Frobenius homomorphism}. 
It gives a $K$-algebra structure to $F$. 
More precisely, the scalar multiplication of $K$ on $F$ is given by $x\cdot a=x^pa$ for $x\in K$ and $a\in F$.
If $A$ is a $K$-algebra, then seeing $F$ as a $K$-algebra via the Frobenius map, we can form the tensor product over $K$ of these two $K$-algebras, $\Frob_{K/F}A=F\otimes_KA$, which is then an $F$-algebra, where for $\lambda\in K$, $\beta\in F$ and $a\in A$ we have $\beta\otimes\lambda a=\lambda^p\beta\otimes a$.

\begin{prop}\label{properties-Frob-alg}
    Let $K/F$ be a purely inseparable field extension with $K^p\subseteq F$. The following hold:
\begin{enumerate}[$(1)$]
    \item If $A$ and $B$ are $K$-algebras, then we have
    $$\Frob_{K/F}(A\otimes_KB)\simeq \Frob_{K/F}A\otimes_F\Frob_{K/F}B.$$
    \item For $n\in\mathbb{N}_+$, we have that $\Frob_{K/F}\mathbb{M}_n(K)\simeq\mathbb{M}_n(F)$.
    \item If $A$ is a central simple $K$-algebra, then $\Frob_{K/F}A$ is a central simple $F$-algebra.
\end{enumerate}	
\end{prop}
\begin{proof} 
    See \cite[Section 3.13, p.149]{Jacob1996}.
\end{proof}

Let $K/F$ be a purely inseparable field extension with $K^p\subseteq F$. By \Cref{properties-Frob-alg}, we obtain a group homomorphism
\begin{equation*}
    \Frob_{K/F}:\Br_p(K)\to\Br_p(F), [A]\to[\Frob_{K/F}A].
\end{equation*}
Recall further the group homomorphism $r_{K/F}:\Br_p(F)\to\Br_p(K)$ given by $r_{K/F}[A]=[A\otimes_FK]$ which is induced by the scalar extension.

\begin{prop}\label{r+frob=complex-p-torsion}
    Let $K/F$ be a purely inseparable field extension with $K^p\subseteq F$. The following hold:
\begin{enumerate}[$(1)$]
    \item For $x\in K$ and $y\in K^{\times}$, we have that $\Frob_{K/F}([x,y)_{p,K})=[x^p,y^p)_{p,F}$.
    \item The composite homomorphism $\Frob_{K/F}\circ r_{K/F}:\Br_p(F)\to\Br_p(F)$ is trivial. 
\end{enumerate}	
\end{prop}
\begin{proof}	
    $(1)$ Set $C=[x,y)_{p,K}$ with $x\in K$, $y\in K^{\times}$. 
    Let $i,j\in C\setminus K$ be such that $i^p-i=x$, $j^p=y$ and $ji=(i-1)j$. 
    Now ($1\otimes1$, $1\otimes i$, $1\otimes j$, $1\otimes ij$) is an $F$-basis for $\Frob_{K/F}C$, and we have that
\begin{equation*}
\begin{split}
    (1\otimes i)^p-1\otimes i &=1\otimes(i^p-i)=1\otimes x=x^p(1\otimes1),\\
    (1\otimes j)^p &= 1\otimes y=y^p(1\otimes1),\\
    (1\otimes j)(1\otimes i)-1\otimes j&=1\otimes(ji-j)=1\otimes ij=(1\otimes i)(1\otimes j).
\end{split}
\end{equation*}	 
    Therefore $\Frob_{K/F}C\simeq[x^p,y^p)_{p,F}$.
	
    $(2)$ Let $a\in F$ and $b\in F^{\times}$. 
    We have that $\Frob_{K/F}([a,b)_{p,F}\otimes_FK)=[a^p,b^p)_{p,F}$, which is split, by \Cref{cyclic-alg-properties}.
    Since, by \Cref{Albert-Merkur-Suslin}, $\Br_p(F)$ is generated by the classes of $F$-algebras $[a,b)_{p,F}$ with $a\in F$ and $b\in F^{\times}$, this shows the statement.
\end{proof}
The assertion $(2)$ of \Cref{r+frob=complex-p-torsion} can be also obtained by the fact that $\Frob_{F/F}$ induces multiplication by $p$ on $\Br(F)$, see \cite[Theorem 4.1.2]{Jacob1996}.

\section{Cyclic algebras of exponent $p$ in characteristic $p$}
In this section, we show that in characteristic $p$, the Brauer class of any cyclic algebra of exponent $p$ and degree $p^n$, with $n\in\mathbb{N}_+$, is given by a tensor product of $p^{n-1}$ cyclic algebras of degree $p$. 
This bound was first obtained by Mammone and Merkurjev (\Cref{MM-cyclic-p-symbol}), based on a computation of the corestriction of cyclic algebras. 
In our argument, we make use of the Frobenius morphism. 

The following is a generalization of the result \cite[Proposition 3.3]{BBL2022}.

\begin{prop}\label{C:multiquad-split-quaternion-char2}
    Let $L/F$ be a finite field extension and let $C$ be a cyclic $L$-algebra of degree $p$.
    Then there exists a purely inseparable field extension $K/F$ with $[K:F]\leq p^{[L:F]}$ such that $C_{LK}$ is split. 
    Moreover, $K/F$ can be chosen such that $K^p\subseteq F$ when $L/F$ is separable.
\end{prop}
\begin{proof}
    We have $C\simeq[x,y)_{p,L}$ for certain $x\in L$ and $y\in L^{\times}$.
    By \cite[Lemma 2.3]{BBL2022}, there exists a purely inseparable field extension $K/F$ such that $y\in {K(y)}^{\times p}$ and $[K:F]\leq p^{[L:F]}$. 
    Then $C_{LK}$ is split by \Cref{cyclic-alg-properties}.
    When $L/F$ is separable, it  follows further by \cite[Lemma 2.3]{BBL2022}  that $K^p\subseteq F$. 
\end{proof}

For a given central simple $F$-algebra $A$, we denote by $A^{\op}$ the \emph{opposite algebra of} $A$, which is defined by endowing the additive group $A$ with the multiplication $\cdot^{\op}$ given by $a\cdot^{\op} b=b\cdot a$ for $a,b\in A$, where $\cdot$ is the multiplication in $A$.

The following result is obtained by Mammone and Moresi \cite[Théorème 4]{MamMor95} for $p=2$. 
Although their argument also makes use of the Frobenius morphism, our argument is different. 
\begin{prop}\label{multipurinsep-ind-reduce to p+cyclic}
    Let $A$ be a central simple $F$-algebra with $\exp A=p$ and $n\in\mathbb{N}_+$. 
    Let $K/F$ be a purely inseparable field extension with $K^p\subseteq F$ and \linebreak $[K:F]=p^n$. 
    Assume $A_K$ is Brauer equivalent to a cyclic $K$-division algebra of degree $p$. Then $\lambda_p(A)\leq n+p-1$.
    More precisely, $A\sim B\otimes_FB'$ for some central simple $F$-algebras $B$ and $B'$ such that $B_K$ is split and $\lambda_p(B')\leq p-1$.
\end{prop}
\begin{proof} 
    We have $A_K\sim[x,y)_{p,K}$ for some $x\in K$ and $y\in K^{\times}$. 
    Let $i,j\in [x,y)_{p,K}$ be such that $i^p-i=x$, $j^p=y$ and $ji=(i-1)j$. Set $L'=K(i)$. 
    Then, by \cite[Lemma 7.7]{Alb39}, we have that $L'=KL$ with $L=F(i^p)$ and $L/F$ is cyclic of degree $p$.
    It follows by \Cref{r+frob=complex-p-torsion} that $\Frob_{K/F}[x,y)_{p,K}=[x^p,y^p)_{p,F}$ and that $\Frob_{K/F}(A_K)$ is split.  
    Since $\Frob_{K/F}(A_K)\sim\Frob_{K/F}[x,y)_{p,K}$, we get that $[x^p,y^p)_{p,F}$ is split.
    Clearly $L$ can be embedded in $[x^p,y^p)_{p,F}$.
    Hence, we have by \Cref{cyclic-alg-properties} that $y^p=\N_{L/F}(z')$ for some $z'\in L$. 
	
    Assume first that $z'\in F$. Then, we get that $y^p=\N_{L/F}(z')=z'^p$ and therefore $y=z'$ as $\cchar F=p$. 
    Then $A_K\sim[x,z')_{p,K}\simeq[x^p,z')_{p,K}$ and hence $A\otimes_F[x^p,z')^{\op}_{p,F}$ is split over $K$.
    Since $K^p\subseteq F$, it follows by \Cref{Albert-purins->decomp-palg}, that $A\otimes_F[x^p,z')^{\op}_{p,F}\sim C$ where $C$ is a tensor product of $n$ cyclic $F$-algebras of degree $p$. 
    Then $A\sim C\otimes_F[x^p,z')_{p,F}$ and $C$ splits over $K$.
	
    Assume now that $z'\notin F$. As $z'\in L$ and $[L:F]=p$, this implies that $L=F(z')$. 
    Let $$g(X)=X^p+c_{p-1}X^{p-1}+\ldots+c_1X+c_0\in F[X]$$ be the minimal polynomial of $z'$ over $F$. 
    Then $c_0=-\N_{L/F}(z')=(-y)^p$. Set $M=K(\sqrt[p]{c_{1}},\ldots,\sqrt[p]{c_{p-1}})$. 
    Then $M/K$ is a purely inseparable field extension with $M^p\subseteq F$ and $[M:K]\leq p^{p-1}$. 
    We claim that $z'=z^p$ for some $z\in ML$. 
    Since $K/F$ is purely inseparable and $L=F(z')$, we have $KL=K(z')$ and $g$ is the minimal polynomial of $z'$ over $K$. 
    Now the polynomial 
\begin{equation*}
    g'(X)=X^p+\sqrt[p]{c_{p-1}}X^{p-1}+\ldots+\sqrt[p]{c_1}X-y\in M[X]
\end{equation*}    
    satisfies $g'(\sqrt[p]{z'})=0$. Hence $[M(\sqrt[p]{z'}):M]\leq p$. 
    Clearly $ML=M(z')\subseteq M(\sqrt[p]{z'})$. Since $[ML:M]=p$, we deduce that $ML=M(z')=M(\sqrt[p]{z'})$.
    Hence $z'=z^p$ with $z\in ML$ as claimed. Now $M/F$ is purely inseparable, hence $\N_{L/F}(z')=\N_{ML/M}(z')$. Thus we get that
\begin{equation*}
    y^p=\N_{L/F}(z')=\N_{ML/M}(z^p)=\N_{ML/M}(z)^p
\end{equation*}     
    and as $\cchar F=p$, we obtain $y=\N_{ML/M}(z)$. 
    Since $A_{M}\sim[x,y)_{p,M}$ and $y=\N_{ML/M}(z)$ with $z\in ML$, it follows by \Cref{cyclic-alg-properties} that $A_{M}$ is split. 
    If $b_1,\ldots,b_n\in F^{\times}$ are such that $K=F(\sqrt[p]{b_1},\ldots,\sqrt[p]{b_n})$, then we have 
    $M=F(\sqrt[p]{b_1},\ldots,\sqrt[p]{b_n},\sqrt[p]{c_1},\ldots,\sqrt[p]{c_{p-1}})$.
    As $A_M$ is split, it follows by \Cref{Albert-purins->decomp-palg} that $A\sim B\otimes_FB'$ where $B=\bigotimes_{i=1}^n[a_i,b_i)_{p,F}$ and 
    $B'=\bigotimes_{i=1}^{p-1}[e_i,c_i)_{p,F}$ for some $a_1,\ldots,a_n,e_1,\ldots,e_{p-1}\in F$. 
    Clearly $B$ splits over $K$ and this concludes the proof. 
\end{proof}

\begin{cor}\label{symbol-cyclic-deg=p^2-exp=p-char=p}
    Let $C$ be a cyclic $F$-algebra with $\exp C=p$ and $\deg C=p^2$. Then $\lambda_p(C)\leq p$.
\end{cor}
\begin{proof}
    As $C$ is cyclic, we can find a purely inseparable field extension $K/F$ of degree $p$ such that $C_K$ is Brauer equivalent to a cyclic $K$-algebra of degree $p$. 
    Now the statement follows directly by \Cref{multipurinsep-ind-reduce to p+cyclic}.
\end{proof}

\begin{rem}
    The bounds in \Cref{symbol-cyclic-deg=p^2-exp=p-char=p} are optimal for $p=2$ and $p=3$.
    The case $p=2$ follows by \Cref{Albert-Tignol-Rowen-2-symbol-small degree}.
    Karpenko's construction in \Cref{generic-indecomp-Karpenko} yields optimality for the case $p=3$. 
    Indeed, if we start with a cyclic algebra $D$ over a field of characteristic $3$ (e.g. a global field of characteristic $3$) with $\deg D=\exp D=9$, then $D(9,3)$ is a cyclic algebra of degree $9$ and exponent $3$ and it does not decompose into a tensor product of two cyclic algebras of degree $3$. 
    This shows the optimality for the case $p=3$. 
\end{rem}

\begin{qu}
    Are the bounds in \Cref{symbol-cyclic-deg=p^2-exp=p-char=p} optimal for $p\geq5$? 
\end{qu}

Mammone and Merkurjev, using the corestriction argument of Tignol from \cite{Tignol1983}, showed that any cyclic $F$-algebra of degree $p^n$ and exponent $p^e$, with $n,e\in\mathbb{N}$, is Brauer equivalent to a tensor product of $p^{n-e}$ cyclic $F$-algebras of degree $p^e$. 
Hence, we retrieve their bounds for $n=2$ and $e=1$. 
To generalize our result to cyclic algebras of exponent $p$ and $p$-power degree, we will first make the following two observations.

\begin{prop}\label{behavior-symbol-length-pur-insep-deg-p}
    Let $A$ be a central simple $F$-algebra with $\exp A=p$. Let $K/F$ be an inseparable field extension with $[K:F]=p$.
    Let $n\in\mathbb{N}_+$ and assume that $\lambda_p(A_K)\leq n$. Then $\lambda_p(A)\leq n\cdot p$.
\end{prop}
\begin{proof}
    Note that, since $[K:F]=p$, we have that $K^p\subseteq F$. 
    By our assumption, we have that $A_K\sim\bigotimes_{i=1}^{n}[x_i,y_i)_{p,K}$ for some $x_1,\ldots,x_n\in K$, $y_1,\ldots,y_n\in K^{\times}$.
    If $y_i\in F^{\times}$ for all $i=1,\ldots,n$, then as $[x_i,y_i)_{p,K}\simeq[x_i^p,y_i)_{p,K}$, it follows that $A\otimes_F\bigotimes_{i=1}^{n}[x_i^p,y_i)^{\op}_{p,F}$ splits over $K$, and we obtain that $\lambda_p(A)\leq n+1<n\cdot p$.

    Assume now, without loss of generality, that $y_1\notin F$. Set $y=y_1$.  
    As $[K:F]=p$, we get that $K=F(y)$. For $i=2,\ldots,n$, we write $y_i=\sum_{j=0}^{p-1} c_{i,j}y^j$ with $c_{i,0},\ldots,c_{i,p-1}\in F$. We let 
\begin{equation*}
    M=K(\sqrt[p]{c_{2,0}},\ldots,\sqrt[p]{c_{2,p-1}},\ldots,\sqrt[p]{c_{n,0}},\ldots,\sqrt[p]{c_{n,p-1}}).
\end{equation*}    
    Note that $M/F$ is a purely inseparable field extension with  $M^p\subseteq F$ and we have $[M:F]\leq p^{(n-1)\cdot p+1}$.
    Then $A$ splits over $M(\sqrt[p]{y})$ and therefore $\lambda_p(A_M)\leq1$. 
    Now it follows by \Cref{multipurinsep-ind-reduce to p+cyclic} that $\lambda_p(A)\leq (n-1)\cdot p+1+p-1=n\cdot p$.
\end{proof}

\begin{cor}\label{behavior-symbol-length-pur-insep}
    Let $A$ be a central simple $F$-algebra with $\exp A=p$. 
    Let $K/F$ be a purely inseparable field extension and assume that $A_K$ is not split. Then $\lambda_p(A)\leq[K:F]\cdot\lambda_p(A_K)$.
\end{cor}
\begin{proof}
    We prove the statement by induction on $[K:F]$, the case $[K:F]=1$ being trivially true. Assume now that $[K:F]>1$. 
    Note that, as $K/F$ is a purely inseparable field extension, we have that $[K:F]$ is a $p$-power. 
    Now, we can find a subfield $K'/F$ of $K/F$ with $[K':F]=p$.  
    It follows by induction hypothesis that $\lambda_p(A_{K'})\leq[K:K']\cdot\lambda_p(A_K)$.
    Then, using \Cref{behavior-symbol-length-pur-insep-deg-p}, we obtain that $\lambda_p(A)\leq p\cdot\lambda_p(A_{K'})\leq p\cdot[K:K']\cdot\lambda_p(A_K)=[K:F]\cdot\lambda_p(A_K)$.
\end{proof}

As a consequence, we obtain bounds for the symbol length of cyclic algebras of exponent $p$. 
These bounds coincide with those in \Cref{MM-cyclic-p-symbol}.
\begin{thm}\label{symbol-length-cyclic-exp=p-char=p}
    Let $C$ be a cyclic $F$-algebra with $\exp C=p$ and $\deg C=p^n$ with $n\in\mathbb{N}_+$. 
    Then $\lambda_p(C)\leq p^{n-1}$.
\end{thm}
\begin{proof}
    As $C$ is cyclic, we can find a purely inseparable field extension $K/F$, with $[K:F]=p^{n-1}$, contained in $C$.
    Then $C_{C}(K)$ is a cyclic $K$-algebra of degree $p$, which is Brauer equivalent to $C_{K}$. 
    It follows by \Cref{behavior-symbol-length-pur-insep} that $\lambda_p(C)\leq p^{n-1}\cdot \lambda_p(C_K)=p^{n-1}$.
\end{proof}

\section{Algebras of exponent $p$ in characteristic $p$} 
In this section, we show that in characteristic $2$, the Brauer class of a central simple algebra of exponent $2$ and index $2^n$, with $n\in\mathbb{N}_+$, is given by a tensor product of $2^n-1$ quaternion algebras. 
This bound was first obtained by Florence using a generic splitting field given by the function field of the Severi-Brauer variety attached to the algebra (\Cref{Flo-p-symbol}). 
We study purely inseparable splitting fields and make use of the Frobenius morphism.

To capture an essential hypothesis for the method that we are going to present, we introduce a condition on the field $F$, which might be trivially satisfied, depending on the prime $p$ given by the characteristic of $F$.
We call the field $F$ \emph{$p$-reducible} (respectively \emph{$p$-cyclic reducible}) if, for every central simple $F$-algebra $A$ of exponent $p$ and every purely inseparable field extension $K/F$ with $K^p\subseteq F$, either $A_K$ is split or
there exists a separable field extension $L/F$ of degree $\frac{1}{p}\ind A_K$ such that $\ind A_{LK}=p$ (respectively $A_{LK}$ is Brauer equivalent to a cyclic $LK$-algebra of degree $p$).
Note that these two notions coincide for prime numbers $p$ where central division algebras of degree $p$ are cyclic, in particular for $p=2$ and $p=3$ \cite[Theorem 11.5]{Alb39}. 

\begin{exmp}
    If $p=2$, then $F$ is $2$-cyclic reducible. 
    Indeed, $F$ is $2$-reducible by \cite[Lemma 4]{Bec16} and any central division algebra of degree $2$ is cyclic.
\end{exmp}  
We don't know whether fields of characteristic $p>2$ are generally $p$-cyclic reducible, or even $p$-reducible.

\begin{lem}\label{process-for-exist-multipurins splitting-p-alg-exp=p}
    Assume that $F$ is $p$-cyclic reducible. Let $A$ be a central simple $F$-algebra with $\exp A=p$. 
    Let $K/F$ be a finite purely inseparable field extension with $K^p\subseteq F$ such that $\ind A_K=p^n$ with $n\in\mathbb{N}_+$. 
    There exists a purely inseparable field extension $K'/K$ with $K'^p\subseteq F$ and $[K':K]\leq p^{(p-1)\cdot p^{n-1}}$ such that $\ind A_{K'}\leq p^{n-1}$.
\end{lem}
\begin{proof}
    By our assumption, there exists a separable field extension $L/F$ with $[L:F]=p^{n-1}$ such that $A_{LK}$ is Brauer equivalent to a cyclic $LK$-algebra of degree $p$. 
    It follows by \Cref{multipurinsep-ind-reduce to p+cyclic} that $A_L\sim B\otimes_LB'$ for some central simple $L$-algebras $B$ and $B'$ such that $B$ splits over $LK$ and $B'$ is a tensor product of $(p-1)$ cyclic $L$-algebras of degree $p$.
    By \Cref{C:multiquad-split-quaternion-char2}, there exist $b_1,\ldots,b_m\in F^{\times}\setminus F^{\times p}$ with $m\leq (p-1)\cdot p^{n-1}$ such that $L(\sqrt[p]{b_1},\ldots,\sqrt[p]{b_m})$ splits $B'$. 
    Consider $K'=K(\sqrt[p]{b_1},\ldots,\sqrt[p]{b_m})$. 
    Then $K'/F$ is a purely inseparable field extension with $K'^{p}\subseteq F$ and $[K':K]\leq p^{(p-1)\cdot p^{n-1}}$. 
    Now $LK'$ splits $B'$ and $B$, and hence $B\otimes_LB'$.
    Therefore $A_{LK'}$ is split. Since $[LK':K']=[L:F]=p^{n-1}$, we obtain that $\ind A_{K'}\leq p^{n-1}$.
\end{proof}

\begin{prop}\label{purins of exp=p-reducing-ind-palg-exp=p}
    Assume that $F$ is $p$-cyclic reducible.
    Let $A$ be a central simple $F$-algebra with $\exp A=p$ and $\ind A=p^n$ with $n\in\mathbb{N}_+$. 
    Then, for $1\leq i\leq n$, there exist purely inseparable field extensions $K_i/F$ with $K_i^p\subseteq F$ and $K_{i-1}\subseteq K_i$ such that $\ind A_{K_i}\leq p^{n-i}$ and that $[K_i:F]\leq p^{2\cdot p^{n-1}-p^{n-i}}$.
\end{prop}
\begin{proof}
    By our assumption, there exists a separable field extension $L/F$ such that $[L:F]=p^{n-1}$ and that $A_{L}$ is Brauer equivalent to a cyclic $L$-algebra of degree $p$. 
    By \Cref{C:multiquad-split-quaternion-char2}, we can find a purely inseparable field extension $K_1/F$ with $K_1^p\subseteq F$ and $[K_1:F]\leq p^{p^{n-1}}$ such that $A_{LK_1}$ is split. In particular, $\ind A_{K_1}\leq p^{n-1}$.
    Then, a repeated use of \Cref{process-for-exist-multipurins splitting-p-alg-exp=p}, starting by applying to $A_{K_1}$, yields purely inseparable field extensions $K_i/F$ with $K_{i-1}\subseteq K_i$, $K_i^p\subseteq F$ and $[K_i:K_{i-1}]\leq p^{(p-1)\cdot p^{n-i}}$ such that $\ind A_{K_i}\leq p^{n-i}$ for $2\leq i\leq n$. 
    We have 
\begin{equation*}
    [K_i:F]\leq p^{p^{n-1}}\cdot p^{(p-1)\cdot p^{n-2}}\cdots p^{(p-1)\cdot p^{n-i}}=p^{2\cdot p^{n-1}-p^{n-i}}
\end{equation*}      
    for $2\leq i\leq n$.
\end{proof}

\begin{prop}\label{symbol-length-exp p-char p-under assumptions}
    Assume that $F$ is $p$-cyclic reducible.
    Let $A$ be a central simple $F$-algebra with $\exp A=p$ and $\ind A=p^n$ with $n\in\mathbb{N}_+$. 
    Then $\lambda_p(A)\leq2\cdot p^{n-1}-1$.
\end{prop}
\begin{proof}
    \Cref{purins of exp=p-reducing-ind-palg-exp=p} yields a purely inseparable field extension $K/F$ with $K^p\subseteq F$ and $[K:F]\leq p^{2\cdot p^{n-1}-1}$ such that $A_K$ is split. 
    Now the statement follows by \Cref{Albert-purins->decomp-palg}. 
\end{proof}

\begin{cor}\label{3-symbol-for-3-red-fields}
    Let $p=3$ and assume that $F$ is $3$-reducible. 
    Let $A$ be a central simple $F$-algebra with $\exp A=3$ and $\ind A=3^n$ with $n\in\mathbb{N}_+$.
    Then $\lambda_3(A)\leq2\cdot3^{n-1}-1$.
\end{cor}
\begin{proof}
    Since any division algebra of degree $3$ is cyclic, the statement follows by \Cref{symbol-length-exp p-char p-under assumptions}.
\end{proof}

Let us close this section with our result on the $2$-symbol length in characteristic $2$. 
The bounds coincide with the ones in \Cref{Flo-p-symbol}
\begin{thm}\label{2-symbol-length-char=2}
    Let $p=2$. Let $A$ be a central simple $F$-algebra with $\exp A=2$ and $\ind A=2^n$ with $n\in\mathbb{N}_+$. Then $\lambda_2(A)\leq2^{n}-1$.
\end{thm}
\begin{proof}
    Since $F$ is $2$-cyclic reducible, the statement follows by \Cref{symbol-length-exp p-char p-under assumptions}.
\end{proof}

\section{Splitting by $p$-extensions}\label{section:Galois-p-elemantary-split}
In this section, we address \Cref{p-symbol-p-elemantary-crossed-product?}.
We provide bounds on the $p$-symbol length under the extra assumption that the considered algebra has a splitting field which is a $p$-extension.
Let us start with the following observation.
\begin{prop}\label{behavior-p-symbol-length-cyclic-ext-deg=p}
    Let $A$ be a central simple $F$-algebra with $\exp A=p$. Let $L/F$ be a cyclic field extension with $[L:F]=p$. 
    Let $n\in\mathbb{N}$ and assume that $\lambda_p(A_L)\leq n$. Then $\lambda_p(A)\leq (n+1)\cdot p-1$.
\end{prop}
\begin{proof}
    Since $L/F$ is a separable field extension and $\lambda_p(A_L)\leq n$, by \Cref{C:multiquad-split-quaternion-char2}, there exists a purely inseparable field extension $K/F$ with $K^p\subseteq F$ and $[K:F]\leq p^{n\cdot p}$ such that $A_{LK}$ is split. 
    As $LK/K$ is a cyclic field extension of degree $p$, $A_{K}$ is Brauer equivalent to a cyclic $K$-algebra of degree $p$.
    It follows by \Cref{multipurinsep-ind-reduce to p+cyclic} that $\lambda_p(A)\leq n\cdot p+p-1=(n+1)\cdot p-1$.
\end{proof}

\begin{cor}\label{behavior-p-symbol-length-p-ext}
    Let $A$ be a central simple $F$-algebra with $\exp A=p$.
    Let $L/F$ be a $p$-extension. Then $\lambda_p(A)\leq (\lambda_p(A_L)+1)\cdot[L:F]-1$.
\end{cor}
\begin{proof}
    We prove the statement by induction on $[L:F]$, the case $[L:F]=1$ being trivially true. Assume now $[L:F]>1$. 
    Let $L'/F$ be a cyclic field extension of degree $p$ contained in $L/F$. 
    Note that $L/L'$ is a $p$-extension with $[L:L']<[L:F]$. 
    Then, it follows by induction hypothesis that $\lambda_p(A_{L'})\leq (\lambda_p(A_L)+1)\cdot[L:L']-1$. 
    We use \Cref{behavior-p-symbol-length-cyclic-ext-deg=p} and obtain that $\lambda_p(A)\leq (\lambda_p(A_{L'})+1)\cdot p -1\leq(\lambda_p(A_L)+1)\cdot[L:L']\cdot p-1= (\lambda_p(A_L)+1)\cdot[L:F]-1$. 
\end{proof}

\begin{prop}\label{p-symbol-length-Galois-(Z/2)^p}
    Let $A$ be a central simple $F$-algebra with $\exp A=p$.
    Let $n\in\mathbb{N}_+$ and assume that $A$ splits over a $p$-extension of degree $p^n$. 
    Then $\lambda_p(A)\leq2\cdot p^{n-1}-1$. 
\end{prop}
\begin{proof}
    By our assumption, there exists a $p$-extension $L/F$ with $[L:F]=p^{n-1}$ and $\lambda_p(A_L)\leq1$. 
    It follows by \Cref{behavior-p-symbol-length-p-ext} that $\lambda_p(A)\leq(1+1)\cdot[L:F]-1=2\cdot p^{n-1}-1$.
\end{proof}

\Cref{Albert-Tignol-Rowen-2-symbol-small degree} gives a better bound for $p=2$ and $n=3$, we combine this with \Cref{behavior-p-symbol-length-p-ext} to obtain an improvement compared to \Cref{p-symbol-length-Galois-(Z/2)^p} for $p=2$.

\begin{thm}\label{2-symbol-length-Galois-(Z/2)^n}
    Let $p=2$. Let $A$ be a central simple $F$-algebra with $\exp A=2$. 
    Let $n\in\mathbb{N}$ with $n\geq3$ and assume that $A$ splits over a $2$-extension of degree $2^n$. Then $\lambda_2(A)\leq5\cdot2^{n-3}-1$.
\end{thm}
\begin{proof}
    By our assumption, there exists a $2$-extension $L/F$ with $[L:F]=2^{n-3}$ such that $\ind A_L\leq8$.
    Hence by \Cref{Albert-Tignol-Rowen-2-symbol-small degree}, we have that $\lambda_2(A_L)\leq4$. 
    It follows by \Cref{behavior-p-symbol-length-p-ext} that $\lambda_2(A)\leq(4+1)\cdot[L:F]-1=5\cdot 2^{n-3}-1$.
\end{proof}

\section*{Acknowledgments}
This work was supported by the {Fonds Wetenschappelijk Onderzoek – Vlaanderen} in the FWO Odysseus Programme (project G0E6114N, \emph{Explicit Methods in Quadratic Form Theory}), by the {Fondazione Cariverona} in the programme Ricerca Scientifica di Eccellenza 2018 (project \emph{Reducing complexity in algebra, logic, combinatorics - REDCOM}), and by T\"{U}B\.{I}TAK-221N171.

\end{document}